\documentclass[12pt,amstex]{amsart}

\usepackage[top=25mm, bottom=25mm, left=30mm, right=30mm]{geometry}

\usepackage{mathptmx}
\usepackage{mathrsfs}

\usepackage{verbatim}
\usepackage{url}
\usepackage[all]{xy}
\usepackage{color}

\usepackage[colorlinks=true,citecolor=blue]{hyperref}

\usepackage{stmaryrd}
\usepackage{epsfig}
\usepackage{amsmath}
\usepackage{amssymb}

\usepackage{amscd}
\usepackage{graphicx}
\usepackage{pstricks}
\usepackage{tocvsec2}

\newtheorem{thm}{Theorem}[section]
\newtheorem{lem}[thm]{Lemma}

\newtheorem{prop}[thm]{Proposition}

\theoremstyle{definition}

\theoremstyle{remark}

\numberwithin{equation}{section}

\def \N {\mathbb N}

\def \Z {\mathbb Z}
\def \R {\mathbb R}

\def \T {\mathbb{T}}

\def \B {\mathcal{B}}

\def \X {\mathcal{X}}

\def \a {\alpha }
\def \b {\beta}
\def \ep {\epsilon}
\def \d {\delta}

\def\w {\omega}

\begin{document}
\title{A counterexample on multiple convergence without commutativity}

\author{Wen Huang}
\author{Song Shao}
\author{Xiangdong Ye}

\address{School of Mathematical Sciences, University of Science and Technology of China, Hefei, Anhui, 230026, P.R. China}

\email{wenh@mail.ustc.edu.cn}
\email{songshao@ustc.edu.cn}
\email{yexd@ustc.edu.cn}


\thanks{This research is supported by National Natural Science Foundation of China (12371196, 12031019, 12090012, 11971455).}


\begin{abstract}
It is shown that there exist a probability space $(X,\X,\mu)$, two  ergodic measure preserving transformations $T,S$ acting on $(X,\X,\mu)$ with $h_\mu(X,T)=h_\mu(X,S)=0$, and $f, g \in L^\infty(X,\mu)$ such that the limit
\begin{equation*}
  \lim_{N\to\infty}\frac{1}{N}\sum_{n=0}^{N-1} f(T^{n}x)g(S^{n}x)
\end{equation*}
does not exist in $L^2(X,\mu)$. 

\end{abstract}

\maketitle





\section{Introduction}

In the article, the set of integers (resp. natural numbers $\{1,2,\ldots\}$) is denoted by $\Z$ (resp. $\N$). By a {\em measure-preserving system} (m.p.s. for short), we mean a quadruple $(X,\X, \mu, T)$, where $(X,\X,\mu )$ is a Lebesgue probability space and $T: X \rightarrow X$ is an invertible measure preserving transformation.

Let $T,S$ be measure preserving transformations acting on a Lebesgue probability space $(X,\X, \mu)$. Then
for $f,g\in L^\infty(X,\mu)$, the existence in $L^2(X,\mu)$ of the limit
\begin{equation}\label{aa}
  \lim_{N\to\infty}\frac{1}{N}\sum_{n=0}^{N-1} f(T^nx)g(S^nx)
\end{equation}
was established in the commutative case by Conze and Lesigne \cite{CL84}, and the extension to
any finite number of transformations spanning a nilpotent group and to polynomial iterates
was established by Walsh \cite{Walsh}. When $T,S$ do not generate a nilpotent group, then \eqref{aa} may not exist \cite{BL02}.
Moreover, if both $T$ and $S$ have positive entropy, \eqref{aa} may fail to exist \cite[Proposition 1.4]{FranHost21}.
Recently,  Frantzikinakis and Host  \cite{FranHost21} gave the following beautiful multiple convergence theorem without commutativity.

\medskip

\noindent {\bf Theorem F-H.}
{\em
Let $T,S$ be measure preserving transformations acting on a probability
space $(X,\X,\mu)$ such that the system $(X,\X,\mu,T)$ has zero entropy. Let also $p\in \Z[t]$ be a
polynomial with $\deg (p) \ge  2$. Then for every $f,g\in L^\infty(X,\mu)$, the limit
\begin{equation}\label{bb}
  \lim_{N\to\infty}\frac{1}{N}\sum_{n=0}^{N-1} f(T^nx)g(S^{p(n)}x)
\end{equation}
exists in $L^2(X,\mu)$.
}

\medskip

It is unknown that whether {Theorem F-H} holds  when one replaces the iterates $n,p(n)$ by the pair $n,n$ or $n^2,n^3$ or
arbitrary polynomials $p,q\in \Z[t]$ with $p(0) = q(0) = 0$ (except the case covered by the above theorem), see \cite[Problem]{FranHost21} and the sentence below it. In \cite{HSY2023}, we showed that for polynomials $p_1, p_2\in \Z[t]$ with $\deg p_1, \deg p_2\ge 5$ there exist a probability space $(X,\X,\mu)$, two  ergodic measure preserving transformations $T,S$ acting on $(X,\X,\mu)$ with $h_\mu(X,T)=h_\mu(X,S)=0$, and $f, g \in L^\infty(X,\mu)$ such that the limit
\begin{equation*}
  \lim_{N\to\infty}\frac{1}{N}\sum_{n=0}^{N-1} f(T^{p_1(n)}x)g(S^{p_2(n)}x)
\end{equation*}
does not exist in $L^2(X,\mu)$. In this paper, we show the case $p_1(n)=p_2(n)=n$.


\medskip

To be precise, the main result of the paper is as follows.

\medskip

\noindent {\bf Main Theorem.}
{\em there exist a probability space $(X,\X,\mu)$, two  ergodic measure preserving transformations $T,S$ acting on $(X,\X,\mu)$ with $h_\mu(X,T)=h_\mu(X,S)=0$, and $f, g \in L^\infty(X,\mu)$ such that the limit
\begin{equation*}
  \lim_{N\to\infty}\frac{1}{N}\sum_{n=0}^{N-1} f(T^{n}x)g(S^{n}x)
\end{equation*}
does not exist in $L^2(X,\mu)$.
}

\medskip

\subsection*{Remark}
Before submitting our result to arxiv, we got to know that Austin also obtained the counterexample in \cite{Tim}.
Exchanging opinions, we think that it is better to make it separately, as the ideas in the two papers are essentially different,


\section{Proof of the main result}

\subsection{Discrepancy skew product}

For a topological space $X$, we use $\B(X)$ denote the $\sigma$-algebra generated by all open subsets of $X$.

Let $\T=\R/\Z\cong [0,1)$ denote the additive circle. Consider the function $\varphi: \T\rightarrow \{-1,1\}$ defined by
$$\varphi(\theta)=2\cdot 1_{[0,\frac 12)}(\theta)-1=\left\{
                                                      \begin{array}{ll}
                                                        1, & \hbox{$\theta\in [0,\frac 12)$;} \\
                                                        -1, & \hbox{$\theta\in [\frac 12,1)$.}
                                                      \end{array}
                                                    \right.
$$
Let the skew products $T_\a: \T\times \Z\rightarrow \T\times \Z$ be defined for $\a\not\in \mathbb Q$ by
$$T_\a(\theta,v)=(\theta+\a, v+\varphi(\theta)).$$
They are measure preserving transformations of the $\sigma$-finite measure space $(\T\times \Z, \B(\T\times \Z), m_\T\times \#)$, where $m_\T$ is
the Lebesgue measure on $\T$ and $\#$ is counting measure on $\Z$.

We have that for $n\in \N$, $T^n_\a(\theta,v)=(\theta+n\a, v+\varphi_n(\theta))$ where
$$\varphi_n(\theta)=\sum_{k=0}^{n-1}\varphi(\theta+k\a).$$
We call the function $(n,\theta)\mapsto \varphi_n(\theta)$ the {\em discrepancy
cocycle} and $T_\a$ the {\em discrepancy skew product} \cite{Aar2017} (it is called {\em deterministic random walk} in \cite{Aar1982}).

Ergodicity of the discrepancy skew product $T_\a$ was established in \cite{CK76}. For each $n\in \N$, define $\Psi_n: \T\rightarrow \N$ by
\begin{equation}\label{}
  \Psi_n(\theta)=\sum_{k=0}^{n-1} 1_{\T\times \{0\}}\Big(T^k_\a(\theta, 0)\Big)=\#\{0\le k\le n-1: \varphi_k(\theta)=0\}.
\end{equation}
The ergodicity of $T_\a$ tells us that
\begin{equation}\label{b1}
  \Psi_n(\theta)\rightarrow \infty, \ a.e. \quad \text{and}\quad \frac{\Psi_n(\theta)}{n}\rightarrow 0, \ a.e., \ n\to\infty.
\end{equation}

\medskip

We call $\a\not\in \mathbb Q$ {\em badly approximable} if $\inf\{q^2|\a-\frac{p}{q}|: q\in \N, p\in \Z\}>0$. For $a,b>0,M>1$, $a=M^{\pm 1}b$ means $\frac{1}{M}\le \frac{a}{b}\le M$.

\begin{thm}\cite{Aar2017}\label{thm-aar}
If $\a$ is badly approximable, then there is some $M>1$ such that
\begin{equation}\label{}
  \|\Psi_n\|_{L^\infty(\T,m_\T)}\le M\int_\T \Psi_n(\theta) dm_\T(\theta), \quad \text{and} \quad \int_\T \Psi_n(\theta) dm_\T(\theta)=M^{\pm 1}\frac{n}{\sqrt{\log{n}}}.
\end{equation}
\end{thm}
In the case $\a\not\in \mathbb Q$ is quadratic, the result was established in \cite{Aar1982}.

\subsection{The ratio ergodic theorem}


The ratio ergodic theorems, which are generalizations of Birkhoff's ergodic theorem, were given by Hopf and Stepanov for infinite measure preserving transformations \cite{Hopf,Stepanov}, Hurewicz for non-singular transformations\cite{Hurewicz}, and Chacon and Ornstein for operators \cite{CO60}. Hochman studied the ratio ergodic theorem for group actions in \cite{Hochman10, Hochman13}.

\begin{thm}[Hopf-Stepanov ergodic Theorem]\cite{Hopf,Stepanov}
Let $T$ be a conservative, measure preserving transformation of a $\sigma$-finite measure space $(X,\X,\mu)$. If $f$ and $g$ are functions in $L^1$ and $g$ is nonnegative, then the limit
$$\lim_{n\to\infty} \frac{\sum_{k=0}^{n-1}T^kf}{\sum_{k=0}^{n-1}T^kg}$$
exists and is finite almost everywhere on the set $\{x\in X: \sum_{k=0}^\infty T^kg(x)>0\}$.

When $T$ is ergodic and $\int_Xgd\mu\neq 0$,  $\displaystyle \lim_{n\to\infty} \frac{\sum_{k=0}^{n-1}T^kf}{\sum_{k=0}^{n-1}T^kg}=\frac{\int_Xfd\mu}{\int_Xgd\mu}$.
\end{thm}

For any $n\in \N, v\in \Z$, define $\Psi_n^{(v)}: \T\rightarrow \N$ by
\begin{equation}\label{}
  \Psi_n^{(v)}(\theta)=\sum_{k=0}^{n-1} 1_{\T\times \{v\}}\Big(T^k_\a(\theta, 0)\Big)=\#\{0\le k\le n-1: \varphi_k(\theta)=v\}.
\end{equation}
Note that $\Psi_n^{(0)}=\Psi_n$.

In the sequel we fix a badly approximable $\a\not\in \mathbb Q$.
Since  $(\T\times \Z, \B(\T\times \Z), m_\T\times \#, T_\a)$ is ergodic, by Hopf-Stepanov ergodic Theorem for each $v\in \Z$ we have
\begin{equation}\label{b3}
  \frac{\Psi^{(v)}_n(\theta)}{\Psi_n(\theta)}=\frac{\sum_{k=0}^{n-1}1_{\T\times \{v\}}\Big(T^k_\a(\theta, 0)\Big)}{\sum_{k=0}^{n-1}1_{\T\times \{0\}}\Big(T^k_\a(\theta, 0)\Big)}\longrightarrow \frac{\int_{\T\times \Z}1_{\T\times \{v\}}d m_\T\times \#}{\int_{\T\times \Z}1_{\T\times \{0\}}d m_\T\times \#}=1, n\to\infty,
\end{equation}
for $m_\T$-a.e. $\theta\in \T$.
By Egoroff's Theorem, for every $\ep_v>0$, there exists $B_v\in \B(\T)$ such that $m_\T(\T\setminus B_v)<\ep_v$ and $\frac{\Psi^{(v)}_n(\theta)}{\Psi_n(\theta)}\to 1$ uniformly on $B_v$. In particular, there is some $N_v\in \N$ such that $|\Psi^{(v)}_n(\theta)/\Psi_n(\theta)-1|<1$ for all $n\ge N_v$ and $\theta\in B_v$. Thus $\Psi^{(v)}_n(\theta)<2\Psi_n(\theta)$ for all $n\ge N_v$ and $\theta\in B_v$.
By Theorem \ref{thm-aar}, there is some $M_0>1$ such that $ \Psi_n(\theta)\le M_0\frac{n}{\sqrt{\log{n}}}$ for all $n\in \N$ and $m_\T$-a.e. $\theta\in \T$. It follows that $\Psi^{(v)}_n(\theta)<2\Psi_n(\theta)\le 2M_0\frac{n}{\sqrt{\log{n}}}$ for all $n\ge N_v$ and $m_\T$-a.e. $\theta\in B_v$.
Without loss of generality, we assume that  $\Psi^{(v)}_n(\theta) < 2M_0\frac{n}{\sqrt{\log{n}}}$ for all $n\ge N_v$ and for all $\theta\in B_v$.
Choose an $M_v>1$ such that $\Psi^{(v)}_n(\theta) \le M_v\frac{n}{\sqrt{\log{n}}}$ for all $n\in \N$ and for all $\theta\in B_v$.

To sum up, we have:
\begin{lem}\label{lem-key}
Let $\a\not\in \mathbb Q$ be badly approximable and $v\in \Z$. For any $\ep_v>0$, there is some $B_v\in \B(\T)$ with $m_\T(\T\setminus B_v)<\ep_v$ and $M_v>1$ such that
\begin{equation}\label{}
  \Psi^{(v)}_n(\theta)\le M_v\frac{n}{\sqrt{\log{n}}}, \ \text{for all}\ n\in \N, \theta\in B_v.
\end{equation}
\end{lem}

\subsection{Construction of the m.p.s. $(X,\X,\mu,T)$}

\subsubsection{}

Let $\T=\R/\Z$ be the circle.
Let $(\T, \B(\T),m_\T, R_\a)$ be the circle rotation by an irrational number $\a$, where $m_\T$ is the Lebesgue measure on $\T$ and
$$R_\a: \T\rightarrow \T,\  \theta\mapsto \theta +\a \pmod{1}.$$

\medskip

Let $\Sigma=\{-1,1\}^\Z$, and $\big(\Sigma, \B(\Sigma), \nu ,\sigma\big)$ be the $(\frac12,\frac 12)$-shift. That is, $\nu=(\frac 12\d_{-1}+\frac12\d_1)^\Z$ is the product measure on $\Sigma=\{-1,1\}^\Z$ and for each sequence $\w\in \Sigma$,
$$(\sigma \w)(n)=\w(n+1), \quad \forall n\in \Z.$$

Now let $(X,\X,\mu)=(\T\times \Sigma, \B(\T\times \Sigma),m_\T\times \nu)$, and define
\begin{equation}\label{}
  T: \T\times \Sigma\rightarrow \T\times \Sigma, \quad (\theta, \w)\mapsto (\theta+\a, \sigma^{\varphi(y)}\w).
\end{equation}
Since $\varphi: \T\rightarrow \Z$ is measurable, it is easy to verify that $(X,\X,\mu,T)$ is a m.p.s.
It is easy to verify that for all $n\ge 1$ and $(\theta,\w)\in X$, we have
\begin{equation}\label{b2}
  T^n(\theta,\w)=(R_\a^n(\theta), \sigma^{\varphi_n(y)}\w)=(\theta+n\a, \sigma^{\varphi_n(\theta)}\w).
\end{equation}

\subsubsection{}

For $a, b\in \mathbb{Z}$ with $a\le b$ and $(s_1,s_2,\ldots, s_{b-a+1})\in \{-1,1\}^{b-a+1}$, let
\begin{equation}\label{}
  _{a}[s_1,s_2,\ldots, s_{b-a+1}]_b:=\{\w\in \Sigma: \w(a)=s_1,\ldots,\w(b)=s_{b-a+1}\}.
\end{equation}
And for $i\in \{-1,1\}$, $j\in \Z$ let
$$[i]_j:= {_j[i]_j}=\{\w\in \Sigma: \w(j)=i\}.$$

\subsubsection{}
The fact that $(X,\X,\mu,T)$ is ergodic and has zero entropy should be well known. Since we can not find a reference, we give a proof for completeness.

\begin{prop}\label{prop-zero}
$(X,\X,\mu,T)$ is an ergodic m.p.s. with $h_\mu(X,T)=0$.
\end{prop}

\begin{proof}
First we show that $(X,\X,\mu,T)$ is ergodic. Let $A_1,A_2\in \B(\T)$ with $m_\T(A_1)m_\T(A_2)>0$ and $B_j={_{-M_j}}[s^j]_{M_j}\in \B(\Sigma)$, where $s^j\in \{0,1\}^{2M_j+1}$
with some $M_j\in \N$, $j=1,2$.
By \eqref{b2},
\begin{equation*}
  \begin{split}
     & \mu\big((A_1\times B_1)\cap T^{-n}(A_2\times B_2)\big) \\
     = & \int_\Sigma \int_\T 1_{R_\a^{-n}A_2\cap A_1}(\theta)\cdot 1_{\sigma^{-\varphi_n(\theta)}B_2\cap B_1(\w)} dm_\T(\theta)d\nu(\w)\\
= & \int_{\T} 1_{R_\a^{-n}A_2\cap A_1}(\theta)\cdot \nu({\sigma^{-\varphi_n(\theta)}B_2\cap B_1}) dm_\T(\theta).
  \end{split}
\end{equation*}
Note that when $|q|>M_1+M_2$, $\nu(\sigma^{-q}B_1\cap B_2)=\nu(B_1)\nu(B_2)$. Thus
\begin{equation}\label{}
  \begin{split}
     & \mu\big((A_1\times B_1)\cap T^{-n}(A_2\times B_2)\big) \\
\ge & \int_{\T} 1_{R_\a^{-n}A_2\cap A_1}(\theta)\cdot \nu({\sigma^{-\varphi_n(\theta)}B_2\cap B_1}) 1_{\{k\in \N:|\varphi_k(\theta)|>M_1+M_2\}}(n) dm_\T(\theta)\\
=& \int_{\T} 1_{R_\a^{-n}A_2\cap A_1}(\theta)\cdot \nu(B_1)\nu(B_2)1_{\{k\in \N:|\varphi_k(\theta)|>M_1+M_2\}}(n) dm_\T(\theta)
\\
=& \int_{\T} 1_{R_\a^{-n}A_2\cap A_1}(\theta)\cdot \nu(B_1)\nu(B_2)\big(1-1_{\{k\in \N:|\varphi_k(\theta)|\le M_1+M_2\}}(n)\big) dm_\T(\theta)
\\
\ge & m_\T(R_\a^{-n} A_2\cap A_1)\nu(B_1)\nu (B_2)-\int_{m_\T} 1_{\{k\in \N:|\varphi_k(\theta)|\le M_1+M_2\}}(n) dm_\T(\theta) .
  \end{split}
\end{equation}
Note that
\begin{equation*}
  \begin{split}
     & \frac{1}{N}\sum_{n=0}^{N-1} 1_{\{k\in \N:|\varphi_k(\theta)|\le M_1+M_2\}}(n) = \frac{1}{N}\sum_{n=0}^{N-1}\sum_{v=-(M_1+M_2)}^{M_1+M_2} 1_{\{k\in \N: \varphi_k(\theta)=v\}}(n)\\
      = & \sum_{v=-(M_1+M_2)}^{M_1+M_2} \frac{1}{N}\sum_{n=0}^{N-1} 1_{\{k \in \N: \varphi_k (\theta)=v\}}(n)=\sum_{v=-(M_1+M_2)}^{M_1+M_2} \frac{\Psi_N^{(v)}(\theta)}{N}
   \end{split}
\end{equation*}
Thus
\begin{equation*}
\begin{split}
  & \frac{1}{N}\sum_{n=0}^{N-1} \mu\big((A_1\times B_1)\cap T^{-n}(A_2\times B_2)\big)\\ \ge &  \frac{1}{N}\sum_{n=0}^{N-1} m_\T(R_\a^{-n} A_2\cap A_1)\nu(B_1)\nu (B_2)- \sum_{v=-(M_1+M_2)}^{M_1+M_2} \int_{\T}\frac{\Psi_N^{(v)}(\theta)}{N}dm_{\T}(\theta).
\end{split}
 \end{equation*}
By \eqref{b1} and \eqref{b3}, we have for each $v\in \Z$,
$\displaystyle \lim_{N\to\infty}\frac{\Psi_N^{(v)}(\theta)}{N}=0$ for $m_\T$-a.e. $\theta\in \T$. By the Dominated Convergence Theorem, we have for each $v\in \Z$
$$\lim_{N\to\infty}\int_{\T}\frac{\Psi_N^{(v)}(\theta)}{N}dm_\T(\theta)=0.$$
By the ergodicity of $(\T,\B(\T), m_\T,R_\a)$, we deduce
\begin{equation*}
  \begin{split}
     & \lim_{N\to\infty}\frac{1}{N}\sum_{n=0}^{N-1}\mu\big((A_1\times B_1)\cap T^{-n}(A_2\times B_2)\big) \\ \ge & \lim_{N\to\infty} \frac{1}{N}\sum_{n=0}^{N-1} m_\T(R_\a^{-n} A_2\cap A_1)\nu(B_1)\nu (B_2)- \sum_{v=-(M_1+M_2)}^{M_1+M_2} \lim_{N\to\infty} \int_{\T}\frac{\Psi_N^{(v)}(\theta)}{N}dm_{\T}(\theta) \\
= & m_\T(A_1)m_\T(A_2) \nu (B_1)\nu (B_2)=\mu(A_1\times B_1)\mu(A_2\times B_2).
  \end{split}
\end{equation*}
Then it is standard to prove that for all $D_1,D_2\in \X$, we have that
$$\lim_{N\to\infty}\frac{1}{N}\sum_{n=0}^{N-1}\mu\big(D_1\cap T^{-n}D_2\big)\ge \mu(D_1)\mu(D_2). $$
In particular, we have that for any $D_1,D_2\in \X$ with $\mu(D_1)\mu(D_2)>0$, there is some $n\in \N$ such that $\mu\big(D_1\cap T^{-n}D_2\big)>0$, which means that $(X,\X,\mu,T)$ is ergodic.

\medskip

Now we use Abramov-Rokhlin formula to show that $h_\mu(T)=0$.
The proof of this part is almost the same to the proof of Proposition 2.2 in \cite{HSY2023}.

For any finite measurable partition $\b$ of $\Sigma$, we have
\begin{equation*}
  h_\mu(T|R_\a, \b)=\lim_{n\to\infty}\frac{1}{N} \int_\T H_\nu(\bigvee_{n=0}^{N-1}\sigma^{-\varphi_n(\theta)}\b) d m_\T(\theta),
\end{equation*}
where $\varphi_0(\theta)\equiv 0$.

For $\theta\in \mathbb{T}$ and $N\in\N$, we denote $a_N(\theta)$ the cardinality of the set $\{\varphi_n(\theta): 0\le n\le N-1\}$, i.e.
$$a_N(\theta)=|\{\varphi_n(\theta): 0\le n\le N-1\}|.$$

Then $a_N$ is a measurable function from $\mathbb{T}$ to $\{1,2,\cdots,N\}$, and the cardinality of $\bigvee_{n=0}^{N-1}\sigma^{-\varphi_n(\theta)}\b$ is not greater than $|\b|^{a_N(\theta)}$ for any $\theta\in \mathbb{T}$, and hence
$$\frac{1}{N} \int_\T H_\nu(\bigvee_{n=0}^{N-1}\sigma^{-\varphi_n(\theta)}\b) d m_\T(\theta)\le \frac{1}{N} \int_\T \log |\b|^{a_N(\theta)}d m_\T(\theta)= \int_\T \frac{a_N(\theta)}{N}\log |\b| d m_\T(\theta).$$
We claim that
for $m_\T$-a.e. $\theta\in \T$,
\begin{equation}\label{c3}
  \lim_{N\to\infty} \frac{a_N(\theta)}{N}=0.
\end{equation}
We now show the claim. Since $(\T,\B(\T),m_\T,R_\a)$ is ergodic and $\int_\T \varphi dm=0$, by Birkhoff ergodic theorem, for $m_\T$-a.e. $\theta\in \T$,
$$\frac{\varphi_n(\theta)}{n}=\frac{1}{n}\sum_{k=0}^{n-1}\varphi(R^k_\a \theta)\to \int_\T \varphi d m_\T=0, \ n\to\infty.$$
Thus for $m_\T$-a.e. $\theta\in \T$, for any $\ep>0$ there is some $M(\theta,\ep)\in \N$ such that when $ n\ge M(\theta,\ep)$, we have
$|\frac{\varphi_n(\theta)}{n}|\le \ep$. Thus for $N>M(\theta,\ep)$, we have $|\varphi_n(\theta)|\le \ep n\le \ep N$ for all $M(\theta,\ep)\le n\le N$. It follows that
$a_N(\theta)\le M(\theta,\ep)+2\ep N+1,$
and
$$ \lim_{N\to\infty} \frac{a_N(\theta)}{N}\le  \lim_{N\to\infty} \frac{M(\theta,\ep)+2\ep N+1}{N}\le 2\ep.$$
Since $\ep$ is arbitrary, we have \eqref{c3}, i.e. for $m_\T$-a.e. $\theta\in \T$,
$\lim \limits_{N\to\infty} \frac{a_N(\theta)}{N}=0.$ This ends the proof of the claim.

Thus by the Dominated Convergence Theorem,
\begin{align*}
h_\mu(T|R_\a, \b)&=\lim_{N\to\infty}\frac{1}{N} \int_\T H_\nu(\bigvee_{n=0}^{N-1}\sigma^{-\varphi_n(\theta)}\b) d m_\T(\theta)\\
&\le \log |\b| \lim_{N\to\infty}\int_\T \frac{a_N(\theta)}{N}d m_\T(\theta)=\log |\b|\int_\T \lim_{N\to\infty}\frac{a_N(\theta)}{N}d m_\T(\theta)\\
&=0.
\end{align*}
As $\b$ is an arbitrary finite measurable partition of $\Sigma$, we have $h_\mu(T|R_\a)=0$. Then by Abramov-Rokhlin formula,
$$h_\mu(T)=h_{m_\T}(R_\a)+h_\mu(T|R_\a)=0.$$
The proof is complete.
\end{proof}

\subsection{Construction of $(X,\X,\mu,S)$}

\subsubsection{}
We will choose a measurable subset $B\subseteq \T$ with $m_\T(B)>0$ and $E\subseteq \Z$ later. And given such $B\subseteq \T$ and $E\subseteq \Z$, we define a map ${\pi_E}: \Sigma\rightarrow \Sigma$ by
\begin{equation*}\label{}
  ({\pi_E}\w)(s)=\left\{
                     \begin{array}{ll}
                       \w(s), & \hbox{$s\in E$;} \\
                       -\w(s), & \hbox{$s\not\in E$.}
                     \end{array}
                   \right.
\end{equation*}
And define $\pi:\T\times \Sigma \rightarrow \T\times \Sigma$ as follows:
\begin{equation*}
  R(\theta,\w)=(\theta, \pi_E\w).
\end{equation*}
We define a transformation $S: X\rightarrow X$ by
$S:= \pi^{-1}\circ T\circ \pi.$
\begin{equation*}
\xymatrix
{
\T\times \Sigma \ar[d]_{\pi}  \ar[r]^{S}  &  \T\times \Sigma\ar[d]^{\pi} \\
\T\times \Sigma \ar[r]^{T} &  \T\times \Sigma}
\end{equation*}

\subsubsection{Choosing $B$ and $E$}\
\medskip

Let $\a\not\in \mathbb Q$ be badly approximable. By Lemma \ref{lem-key}, for each $v\in \Z$ and any $\ep_v>0$, there is some $B_v\in \B(\T)$ with $m_\T(\T\setminus B_v)<\ep_v$ and $M_v>1$ such that
\begin{equation*}\label{}
  \Psi^{(v)}_n(\theta)\le M_v\frac{n}{\sqrt{\log{n}}}, \ \text{for all}\ n\in \N, \theta\in B_v.
\end{equation*}
Choose $\ep_v$ such that $\sum_{v\in \Z}\ep_v<1$, and let $B=\bigcap_{v\in \Z} B_v$. Then
$$m_\T(B)=m_\T(\bigcap_{v\in \Z}B_v)\ge 1- \sum_{v\in \Z}\ep_v>0.$$
And for all $\theta\in B$, we have
\begin{equation}\label{a1}
  \Psi^{(v)}_n(\theta)\le M_v\frac{n}{\sqrt{\log{n}}}, \ \text{for all}\ n\in \N, v\in \Z.
\end{equation}
For all $v\in \N$, set
\begin{equation}\label{a2}
  C_v=\max\{M_u: -v\le u\le v\}.
\end{equation}
Then $1<C_1<C_2<\cdots$.

Now we define $\displaystyle E=\bigcup_{m=1}^\infty\pm [l_m,l_m+r_m]$ such that
\begin{enumerate}
  \item[(a)] $l_1>1$;
  \item[(b)] for all $m\in \N$, $r_m>l_m$ such that $\displaystyle \frac{C_{l_m}l_m}{\sqrt{\log{(l_m+r_m)}}}<\frac{1}{m}$;
  \item[(c)] for all $m\in \N$, $l_{m+1}>r_m+l_m$ such that $\displaystyle C_{l_m+r_m}\frac{l_m+r_m+1}{\sqrt{\log{l_{m+1}}}}<\frac{1}{m}$.
\end{enumerate}

\subsection{Proof of the Main Theorem}

Let
$$A_1=B\times \Sigma, \quad A_2=A_3=\T\times [1]_0.$$
We will show that the limit
\begin{equation}\label{}
  \frac{1}{N}\sum_{n=0}^{N-1}\mu(A_1\cap T^{-n}A_2\cap S^{-n}A_3)
\end{equation}
does not exist. Then we have our main theorem.

\medskip

Note that for all $n\in \N$, we have
\begin{equation*}
   T^n(\theta,\w)=(\theta+n\a, \sigma^{\varphi_n(\theta)}\w),
\end{equation*}
and
\begin{equation*}
   S^n(\theta,\w)=\pi^{-1}\circ T^n\circ \pi(\theta,\w)=(\theta+n\a, \pi_E^{-1}\circ \sigma^{\varphi_n(\theta)}\circ \pi_E(\w)),
\end{equation*}
It follows that $(\theta, \w)\in A_1\cap T^{-n}A_2\cap S^{-n}A_3$ if and only if
$$\theta\in B, \ \w(\varphi_n(\theta))=1, \ \text{and }\ (\pi_E\w)(\varphi_n(\theta))=1.$$
By the definition of $\pi_E$, we have
$$(\pi_E\w)(\varphi_n(\theta))=1 \Longleftrightarrow
\left\{
  \begin{array}{ll}
    \w(\varphi_n(\theta))=1, & \hbox{if $\varphi_n(\theta)\in E$;} \\
   \w(\varphi_n(\theta))=-1, & \hbox{if $\varphi_n(\theta)\not\in E$.}
  \end{array}
\right.
$$
Thus
$$A_1\cap T^{-n}A_2\cap S^{-n}A_3=\bigcup\{\{\theta\}\times [1]_{\varphi_n(\theta)}\in X: \theta\in B, \varphi_n(\theta)\in E \}.$$
And by Fubini's Theorem
$$\mu(A_1\cap T^{-n}A_2\cap S^{-n}A_3)=\int_{\theta\in B\atop{\varphi_n(\theta)\in E}}\nu([1]_{\varphi_n(\theta)})dm_\T=\frac12m_\T(\{\theta\in \T: \varphi_n(\theta)\in E\}).$$
Hence
\begin{equation}\label{a3}
  \begin{split}
    &  \frac{1}{N}\sum_{n=0}^{N-1}\mu(A_1\cap T^{-n}A_2\cap S^{-n}A_3) \\
     = & \frac 12 \int_B\Big(\frac{1}{N}\sum_{n=0}^{N-1} 1_E(\varphi_n(\theta))\Big)dm_\T\\
    =& \frac 12 \int_B\frac{\#\{0\le n\le N-1: \varphi_n(\theta)\in E\}}{N}dm_\T.
  \end{split}
\end{equation}

\subsubsection{The case when $N=l_{m+1}$}
Let $m\in \N$. Note that for $0\le n\le l_{m+1}-1$, $|\varphi_n(\theta)|\le l_{m+1}-1$, and $[l_m+r_m+1, l_{m+1}-1]\cap E=\emptyset$. Thus we have
\begin{equation*}
  \begin{split}
      & \#\{0\le n< l_{m+1}: \varphi_n(\theta)\in E\} \\
      \le & \#\{0\le n< l_{m+1}: |\varphi_n(\theta)|\le l_m+r_m\} \\
\le & \sum_{k=-(l_m+r_m)}^{l_m+r_m}\Psi_{l_{m+1}}^{(k)}(\theta)\le (2l_m+2r_m+1)C_{l_m+r_m}\frac{l_{m+1}}{\sqrt{\log{l_{m+1}}}}.
   \end{split}
\end{equation*}
The last inequality is from \eqref{a1} and \eqref{a2}. By \eqref{a3}
\begin{equation}\label{a4}
  \begin{split}
    &  \frac{1}{l_{m+1}}\sum_{n=0}^{l_{m+1}-1}\mu(A_1\cap T^{-n}A_2\cap S^{-n}A_3) \\
    =& \frac 12 \int_B\frac{\#\{0\le n\le l_{m+1}-1: \varphi_n(\theta)\in E\}}{l_{m+1}}dm_\T\\
\le & \frac 12 \int_B \frac{(2l_m+2r_m+1)C_{l_m+r_m}}{l_{m+1}}\frac{l_{m+1}}{\sqrt{\log{l_{m+1}}}}dm_\T\\
\le & C_{l_m+r_m}\frac{l_m+r_m+1}{\sqrt{\log{l_{m+1}}}}\le\frac 1m. \quad \quad (\text{by the condition}\ (c))
  \end{split}
\end{equation}

\subsubsection{The case when $N=l_{m}+r_m+1$}
Let $m\in \N$. Since $[l_m,l_m+r_m]\subseteq E$, for $0\le n\le l_m+r_m$, $\varphi_n(\theta)\not\in E$ implies that $\varphi_n(\theta)\le l_m-1$.
Hence we have
\begin{equation*}
  \begin{split}
      & \#\{0\le n\le l_{m}+r_m: \varphi_n(\theta)\in E\} \\
      \ge &l_m+r_m+1- \#\{0\le n\le  l_{m}+r_m: \varphi_n(\theta)\not\in E\} \\
\ge &l_m+r_m+1- \#\{0\le n\le  l_{m}+r_m: |\varphi_n(\theta)|\le l_{m}-1\}\\
\ge &l_m+r_m+1- \sum_{k=-(l_m-1)}^{l_m-1}\Psi_{l_m+r_m}^{(k)}(\theta)\\
\ge & l_m+r_m+1-C_{l_m-1}(2l_m-1)\frac{l_m+r_m}{\sqrt{\log{(l_m+r_m)}}}.
   \end{split}
\end{equation*}
The last inequality is from \eqref{a1} and \eqref{a2}. By \eqref{a3}
\begin{equation}\label{a5}
  \begin{split}
    &  \frac{1}{l_{m}+r_m+1}\sum_{n=0}^{l_{m}+r_m}\mu(A_1\cap T^{-n}A_2\cap S^{-n}A_3) \\
    =& \frac 12 \int_B\frac{\#\{0\le n\le l_{m}+r_m: \varphi_n(\theta)\in E\}}{l_{m}+r_m+1}dm_\T\\
\ge & \frac 12 \int_B\Big( 1-\frac{C_{l_m-1}(2l_m-1)}{l_m+r_m+1}\frac{l_m+r_m}{\sqrt{\log{(l_m+r_m)}}}\Big)dm_\T\\
\ge & \frac{1}{2}\Big(1-2\frac{C_{l_m}l_m}{\sqrt{\log{(l_m+r_m)}}}\Big) \quad \quad \quad (\text{since}\ C_{l_m-1}\le C_{l_m})\\
\ge & \frac 12- \frac{C_{l_m}l_m}{\sqrt{\log{(l_m+r_m)}}}\ge \frac 12-\frac 1m.\quad \quad (\text{by the condition} \ (b))
  \end{split}
\end{equation}
By \eqref{a4} and \eqref{a5}, we have that the limit
\begin{equation*}\label{}
  \frac{1}{N}\sum_{n=0}^{N-1}\mu(A_1\cap T^{-n}A_2\cap S^{-n}A_3)
\end{equation*}
does not exist. The proof of Main Theorem is complete.


\end{document}